\documentclass[11pt]{amsart}
\usepackage{amsmath,amssymb,amsthm}

\theoremstyle{plain}
\newtheorem{theorem}{Theorem}[section]
\newtheorem{proposition}[theorem]{Proposition}
\newtheorem{lemma}[theorem]{Lemma}
\newtheorem{corollary}[theorem]{Corollary}
\newtheorem*{theoremA}{Theorem A}
\theoremstyle{definition}
\newtheorem{example}[theorem]{Example}
\newtheorem{remark}[theorem]{Remark}
\newtheorem{question}[theorem]{Question}

\newcommand{\Oz}{\operatorname{Oz}}
\newcommand{\Aut}{\operatorname{Aut}}
\newcommand{\supp}{\operatorname{supp}}
\newcommand{\kk}{\Bbbk}
\newcommand{\gr}{\operatorname{gr}}
\newcommand{\lf}{\operatorname{lf}}
\newcommand{\rk}{\operatorname{rk}}

\title{The ozone groups of the algebras $B_q(f)$}

\author{James Gómez}
\address{Universidad Pedagógica y Tecnológica de Colombia}
\email{james.gomez@uptc.edu.co}

\author{Helbert Venegas}
\address{Universidad Militar Nueva Granada, Bogot\'a, Colombia}
\email{Helbert.venegas@unimilitar.edu.co}

\subjclass[2020]{16S36, 16S38, 16W22, 16E65}
\keywords{Ozone group, Calabi--Yau algebra, Ore extension, normal element}

\begin{document}

\begin{abstract}
Let $q$ be a primitive $n$-th root of unity, $n>1$, and let $f$ be a nonzero polynomial
such that $n\nmid(j+1)$ for every $j\in\supp(f)$. Set
$e=\gcd(n,\{j+1:j\in\supp(f)\})$. We show that $\Oz(B_q(f))\cong\mu_e\times\mu_e$: the
defining relations are homogeneous for a $\mathbb{Z}/e\times\mathbb{Z}/e$ grading, the
center sits in degree zero, and the ozone group is the character group of that grading. The determination of the ozone group only requires the central elements
 $u^n$, $v^n$, and $\Omega$. The regular normal elements modulo the center form
the same group, generated by $u^{n/e}$ and $v^{n/e}$, so every normal element is central
exactly when $e=1$. For $f=t^2$ we recover a computation of Chan, Gaddis, Won and Zhang,
and for $e>1$ we obtain an infinite family of Calabi--Yau algebras with nontrivial ozone
group.
\end{abstract}

\maketitle

\section{Introduction}

The ozone group of an algebra $A$ is
\begin{align*}
  \Oz(A) &= \{\varphi\in\Aut(A) : \varphi(z)=z \text{ for all } z\in Z(A)\}.
\end{align*}
Chan, Gaddis, Won, and Zhang introduced it in \cite{CGWZ1,CGWZ2}. For noetherian
PI AS-regular algebras the ozone group is trivial
exactly when every normal element is central \cite[Theorem A]{CGWZ2}; and, if an algebra generated in degree one is a skew polynomial ring exactly when the
group is abelian of order $\rk_{Z}(A)$ \cite[Theorem F]{CGWZ2}. Liu, Wu and Zhu
\cite[Theorem 1.2]{LWZ} have shown that the group is always abelian in the PI AS-regular case.

Gaddis and Yee \cite[Definition 1.1]{GY} study the following family. For
$q\in\kk^\times$ and $f\in\kk[t]$,
\begin{align*}
  B_q(f) &= \kk\langle u,v,w \mid uv-qvu,\ wu-quw-f(v),\ wv-q^{-1}vw-f(u)\rangle,
\end{align*}
is the Ore extension $\kk_q[u,v][w;\sigma,\delta]$ with $\sigma(u)=qu$,
$\sigma(v)=q^{-1}v$, $\delta(u)=f(v)$ and $\delta(v)=f(u)$. These are Calabi--Yau
for every $q$ and $f$ \cite[Theorem A]{GY}, and $B_q(t^2)$ is an AS-regular algebra
whose point scheme is a nodal cubic.

We determine the ozone group of this family. For $f(t)=\sum_j c_jt^j$ write
$\supp(f)=\{j:c_j\neq0\}$. Assume throughout that $f\neq0$, that $q$ is a primitive
$n$-th root of unity with $n>1$, and that $n\nmid(j+1)$ for all $j\in\supp(f)$. Put
\begin{align*}
  e &= \gcd\bigl(n,\ \{j+1 : j\in\supp(f)\}\bigr).
\end{align*}

The integer $e$ comes from two separate requirements. Grade $B_q(f)$ by
$\mathbb{Z}/N\times\mathbb{Z}/N$ with $|u|=(1,0)$ and $|v|=(0,1)$; the relation
$uv=qvu$ is homogeneous for any $N$. Writing $|w|=(x,y)$, the relation
$wu=quw+f(v)$ forces $(x+1,y)=(0,j)$ and $wv=q^{-1}vw+f(u)$ forces $(x,y+1)=(j,0)$,
for every $j\in\supp(f)$. Together
\begin{align*}
  |w| &= (-1,-1), & N &\mid (j+1) \quad (j\in\supp(f)).
\end{align*}

Both relations determine $|w|$, and they agree only when $N\mid(j+1)$. The largest
$N$ they allow is $\gcd\{j+1:j\in\supp(f)\}$, which does not see $n$.

Now ask which characters $(\xi_1,\xi_2)$ of the grading can fix the center. Since
$u^n$ and $v^n$ are central of degrees $(n,0)$ and $(0,n)$,
\begin{align*}
  \xi_1^{\,n} &= 1, & \xi_2^{\,n} &= 1,
\end{align*}
so the surviving characters form $\mu_{\gcd(N,n)}\times\mu_{\gcd(N,n)}$. Taking $N$
as large as the relations permit leaves $\mu_e\times\mu_e$, the character group of
the $\mathbb{Z}/e\times\mathbb{Z}/e$ grading. Thus, we may work directly with that grading.

The elements \(u^n\) and \(v^n\) do not, however, exhaust the center, so an additional central element is needed.
Localizing at $u$ and $v$ makes the $\sigma$-derivation inner, and the element
$\gamma$ implementing it turns out to be homogeneous of degree $(-1,-1)$, the degree
of $w$. Then $w'=w-\gamma$ is homogeneous too, and a count of central monomials in
$\kk_q[u^{\pm1},v^{\pm1}][w';\sigma]$ puts the whole center in degree $(0,0)$; this
is Lemma \ref{lem:deg0}. Characters act trivially there, so
\begin{align*}
  u &\mapsto \xi_1u, & v &\mapsto \xi_2v, & w &\mapsto (\xi_1\xi_2)^{-1}w
\end{align*}
is an ozone automorphism for every $(\xi_1,\xi_2)\in\mu_e\times\mu_e$. The main result shows that every ozone automorphism arises in this way.

\begin{theoremA}[Theorem \ref{thm:main}]
$\Oz(B_q(f))\cong\mu_e\times\mu_e$. In particular the ozone group is trivial if and
only if $e=1$, and $|\Oz(B_q(f))|=e^2$.
\end{theoremA}

Theorem 3.13 of \cite{GY} asserts that the ozone group is trivial under the
hypotheses above. The computation there yields $\xi^{\,j+1}=1$ for $j\in\supp(f)$
together with $\xi^n=1$, which leave $\xi\in\mu_e$. Theorem A recovers the stated
conclusion when $e=1$; for $e>1$ the group is nontrivial, the smallest instance
being $n=4$ and $f=t$, worked out in Example \ref{ex:small}. Corollary 3.14 of
\cite{GY} is affected in the same way and is replaced here by
Corollary \ref{cor:normal}.

For the reverse inclusion three central elements are enough. Besides $u^n$ and
$v^n$ there is
\begin{align*}
  \Omega &= uvw+\sum_{j\in\supp(f)}\beta_jc_j\,u^{j+1}-q\sum_{j\in\supp(f)}\alpha_jc_j\,v^{j+1}
\end{align*}
from \cite[Proposition 3.9]{GY}, with scalars $\alpha_j,\beta_j$ recalled in
Section \ref{sec:prelim}. No further information about the center is needed, since none is known \cite[Question 5.3]{GY}. Corollary
\ref{cor:detect} follows inmediately: any automorphism fixing $u^n$, $v^n$ and
$\Omega$ is already an ozone automorphism. Thus, this criterion can be checked directly.

The integer \(e\) also determines precisely when the ozone group is nontrivial. If $e>1$ and $m=n/e$,
then $u^m$ and $v^m$ are normal but not central, and their conjugation
automorphisms generate the ozone group. Note that $e=1$ is weaker than asking $n$
to be coprime to each $j+1$ separately, and that the invariant is a gcd over the
whole support rather than a condition term by term: for $n=6$, both $f=t$ and
$f=t^2$ give a nontrivial ozone group, while $f=t+t^2$ gives a trivial one. See
Example \ref{ex:six}. Every group $\mu_e\times\mu_e$ arises, since $n=2e$ and
$f=t^{e-1}$ satisfy the hypotheses.

 The hypothesis $n\nmid(j+1)$ forces $\gcd(n,j+1)<n$,
hence $e<n$, so $|\Oz(B_q(f))|=e^2$ is a proper divisor of
$\rk_{Z}(B_q(f))=n^2$ and no member of the family is a skew polynomial ring
(Proposition \ref{prop:rank}). The excluded case \(f=0\) corresponds formally to \(e=n\), and \(B_q(0)\) is a skew polynomial ring. Thus, under our hypotheses, \(e\) satisfies \(1\leq e<n\). The case \(e=1\) gives a trivial ozone group, while the limiting case \(e=n\) corresponds to the excluded case $f=0$.

The automorphisms that occur here are the maps $\phi_{a,\xi,h}$ of
\cite[Lemma 4.2]{GY}, recalled in Section \ref{sec:prelim}; Proposition
\ref{prop:phi} identifies those among them that fix the center. For $f=t^2$ one has
$e=\gcd(n,3)$, so
\begin{align*}
  \Oz(B_q(t^2))
  &\cong
  \begin{cases}
    1, & 3\nmid n,\\
    (\mathbb{Z}/3)^2, & 3\mid n,
  \end{cases}
\end{align*}
which is \cite[Proposition 1.8]{CGWZ2}; see Remark \ref{rem:cgwz}.

We use \cite[Lemmas 3.1, 3.2, 3.8, 4.1, 4.2 and Propositions 3.9, 3.10]{GY};
the arguments below are otherwise self-contained. Lemma~\ref{lem:diag} is related
to \cite[Lemma 3.12]{GY}, where the case $d=2$ is treated separately in order to
exclude the coefficient of $w$. Working with leading forms gives the same
conclusion with no case distinction and without the quadratic grading.

\medskip
\noindent\textit{Acknowledgement.} We are grateful to Jason Gaddis for helpful correspondence and to the members of the Algebra Seminar for valuable discussions.

\section{Preliminaries}\label{sec:prelim}

In this section, we collect the preliminary results and conventions that will be used throughout. The exposition follows the framework established by Gaddis and Yee in \cite{GY}, to whom we owe several of the key algebraic identities and lemmas for the family $B_q(f)$.

Throughout, $\kk$ is an algebraically closed field of characteristic zero, $q$ is a primitive $n$-th root of unity with $n>1$, and $f=\sum_jc_jt^j$ is nonzero of degree $d$ with $n\nmid(j+1)$ for all $j\in\supp(f)$. Then $\gcd(n,j+1)<n$ for every $j\in\supp(f)$, so
\begin{align*}
  e &\mid n, & e &< n, & m &:= n/e > 1,
\end{align*}
since $d\in\supp(f)$, then $e\mid (d+1)$. We write $\mu_e$ for the group of $e$-th roots of unity in $\kk^\times$; it is cyclic of order $e$, and we keep the multiplicative notation throughout rather than switching to $\mathbb{Z}/e$.

For $k\geq1$ write $[k]_t=1+t+\cdots+t^{k-1}$, so $[k]_t=0$ exactly when $t$ is a root of unity whose order divides $k$. We use \cite[Lemma 3.1]{GY} in the form
\begin{align*}
  wu^k &= q^ku^kw+\sum_{j=0}^{d}[k]_{q^{j+1}}c_jv^ju^{k-1}, \\
  wv^k &= q^{-k}v^kw+\sum_{j=0}^{d}[k]_{q^{-(j+1)}}c_ju^jv^{k-1}.
\end{align*}
Since $B_q(f)=R[w;\sigma,\delta]$ with $R=\kk_q[u,v]$ is an Ore extension of a domain, $\{u^av^bw^c\}_{a,b,c\geq0}$ is a $\kk$-basis and $B_q(f)$ is a domain. Under the standing hypothesis $u^n,v^n\in Z(B_q(f))$ \cite[Lemma 3.2]{GY}, the algebra is PI and module-finite over its affine center \cite[Proposition 3.10]{GY}, and
\begin{align*}
  \Omega &= uvw+\sum_{j\in\supp(f)}\beta_{j}c_ju^{j+1}
            -q\sum_{j\in\supp(f)}\alpha_jc_jv^{j+1},\\
  \alpha_j&=(q^{-j}-q)^{-1}, \quad \beta_j=(q^{-1}-q^j)^{-1},
\end{align*}
is central \cite[Proposition 3.9]{GY}. The scalars $\alpha_j$ and $\beta_j$ are defined, and nonzero, precisely because $n\nmid(j+1)$; both facts get used. For a regular normal element $a$ we set $\eta_a(b)=a^{-1}ba$, following \cite{CGWZ2}.

\subsection*{Localization.}
Since $u^n$ and $v^n$ are central, the set $X=\{(u^n)^a(v^n)^b : a,b\geq0\}$ is a central
Ore set, and in $\widehat{B}=B_q(f)X^{-1}$ both $u$ and $v$ become invertible:
\begin{align*}
  u^{-1} &= u^{n-1}(u^n)^{-1}, & v^{-1} &= v^{n-1}(v^n)^{-1}.
\end{align*}
So $\widehat{B}$ is the localization of \cite[Lemma 3.8]{GY}, and there is no Ore condition left to check. Two consequences are used repeatedly below. An automorphism of $B_q(f)$ fixing $u^n$ and $v^n$ sends $X$ to itself, hence extends to $\widehat{B}$.
Moreover, an element of $B_q(f)$ commuting with $u$ commutes with $u^{-1}$ as well, and
likewise for $v$, so
\begin{align*}
  Z(B_q(f)) &= Z(\widehat{B})\cap B_q(f).
\end{align*}

Write $R_X=\kk_q[u^{\pm1},v^{\pm1}]$. By \cite[Lemma 3.8]{GY} the derivation $\delta$
becomes inner on $R_X$, implemented by
\begin{align*}
  \gamma &= \sum_{j\in\supp(f)}\alpha_jc_j\,u^{-1}v^j
          - \sum_{j\in\supp(f)}\beta_jc_j\,v^{-1}u^j,
  & \delta(r) &= \gamma r-\sigma(r)\gamma \qquad (r\in R_X).
\end{align*}
Put $w'=w-\gamma$. Then, for $r\in R_X$,
\begin{align*}
  w'r &= \sigma(r)w+\delta(r)-\gamma r = \sigma(r)(w-\gamma) = \sigma(r)w',
\end{align*}
so $\widehat{B}=R_X[w';\sigma]$, and in particular
\begin{align*}
  w'u &= quw', & w'v &= q^{-1}vw', & w'(uv) &= (uv)w'.
\end{align*}

From $uv=qvu$ we get $uvu^{-1}=qv$, while $uvv^{-1}=u$, so
\begin{align*}
  uv\gamma &= q\sum_{j\in\supp(f)}\alpha_jc_j\,v^{j+1}
            - \sum_{j\in\supp(f)}\beta_jc_j\,u^{j+1},
\end{align*}
and therefore
\begin{align*}
  uvw' &= uvw-uv\gamma = \Omega.
\end{align*}
This is the form of $\Omega$ we work with. Centrality is immediate from it: $w'$ commutes
with $uv$, and
\begin{align*}
  \Omega u &= q\,(uvu)\,w' = u\,(uvw') = u\Omega,
  & \Omega v &= q^{-1}uv^2w' = (vuv)w' = v\Omega,
\end{align*}
so $\Omega$ commutes with $R_X$ and with $w'$, hence with $w=w'+\gamma$.

\subsection*{Automorphisms.} Following \cite[(4.1)]{GY}, set
\begin{align*}
  T_q &= \{h\in\kk_q[u,v] : hu=quh,\ hv=q^{-1}vh\}
       = \operatorname{span}_\kk\{u^{bn-1}v^{cn-1} : b,c\geq1\},
\end{align*}
the second description being \cite[Lemma 4.1]{GY}. The computation takes place inside
$\kk_q[u,v]$, so it does not depend on $f$. If $a,\xi\in\kk^\times$ satisfy
\begin{align*}
  a^{d-i} &= 1, & \xi^{i+1} &= 1 && (i\in\supp(f)),
\end{align*}
and $h\in T_q$, then
\begin{align*}
  \phi_{a,\xi,h}(u) &= au, & \phi_{a,\xi,h}(v) &= \xi av,
  & \phi_{a,\xi,h}(w) &= \xi^{-1}a^{d-1}w+h
\end{align*}
determines an automorphism of $B_q(f)$ \cite[Lemma 4.2]{GY}.

We write $R_k=\operatorname{span}_\kk\{u^av^b : a+b=k\}$ and $R_{\leq k}=\bigoplus_{i\leq k}R_i$,
with $R_{\leq k}=0$ for $k<0$; since $uv=qvu$ is homogeneous, $R=\bigoplus_{k\geq0}R_k$ is a
grading. Put
\begin{align*}
  s &= \max(d,1), & F_N &= \sum_{c\geq0}R_{\leq N-sc}\,w^c \quad (N\geq0).
\end{align*}

\noindent
So $F_N$ is spanned by the PBW monomials $u^av^bw^c$ with $a+b+sc\leq N$.

\begin{lemma}\label{lem:filt}
$\{F_N\}$ is an exhaustive algebra filtration of $B_q(f)$, and
\begin{align*}
  \gr B_q(f) &\cong \kk\langle U,V,W\rangle/(UV-qVU,\ WU-qUW,\ WV-q^{-1}VW),
\end{align*}
a skew polynomial ring; in particular $\gr B_q(f)$ is a domain.
\end{lemma}

\begin{proof}

Give the free algebra $\kk\langle u,v,w\rangle$ the weights $\operatorname{wt}(u)=\operatorname{wt}(v)=1$ and $\operatorname{wt}(w)=s$, and let $\widetilde F_N$ be the image in $B_q(f)$ of the span of the words of weight at most $N$. Concatenating words adds weights, so
$\widetilde F_M\widetilde F_N\subseteq\widetilde F_{M+N}$, and $F_N\subseteq\widetilde F_N$ because a PBW monomial is a word. For the reverse inclusion, read the defining relations as rewriting rules
\begin{align*}
  vu &\mapsto q^{-1}uv, & wu &\mapsto quw+f(v), & wv &\mapsto q^{-1}vw+f(u).
\end{align*}
Order the pairs (number of occurrences of $w$, number of inversions) lexicographically. Each rule either removes an occurrence of $w$, in the summands coming from $f$, or keeps their number and removes one inversion; so repeated application to a word terminates, and it terminates in a combination of PBW monomials. No rule raises the weight:
\begin{align*}
  \operatorname{wt}(vu) &= \operatorname{wt}(uv) = 2, \\
  \operatorname{wt}(wu) &= \operatorname{wt}(uw) = s+1, \\
  \operatorname{wt}(v^j) &= j \leq d \leq s < s+1 \qquad (j\in\supp(f)).
\end{align*}

and likewise for the third rule. So a word of weight at most $N$ is a combination of PBW monomials of weight at most $N$, that is $\widetilde F_N=F_N$, and $F_MF_N\subseteq F_{M+N}$.
 The PBW basis shows the filtration is exhaustive, with
\begin{align*}
  \dim_\kk F_N/F_{N-1} &= \#\{(a,b,c) : a+b+sc=N\}.
\end{align*}
Let $S$ be the skew polynomial ring above, weighted by $|U|=|V|=1$ and $|W|=s$. In
$\gr B_q(f)$ the relation $uv=qvu$ gives $UV=qVU$, and $wu=quw+f(v)$ gives $WU=qUW$, since $f(v)\in F_d\subseteq F_s$ while $wu\in F_{s+1}$; similarly $WV=q^{-1}VW$. So $S$ surjects onto $\gr B_q(f)$, and the component of weighted degree $N$ in $S$ has basis $\{U^aV^bW^c : a+b+sc=N\}$, of the same dimension. The surjection is an isomorphism.
\end{proof}

For $f=t^d$ with $d\geq2$ the algebra carries an $\mathbb{N}$-grading, $|u|=|v|=1$ and $|w|=d-1$, and the associated graded algebra for that filtration is $B_q(t^d)$ itself: the term $f$ survives. This is no good for Lemma \ref{lem:diag}, whose argument reduces modulo the ideal generated by the image of $u$, and $u$ is not normal in $B_q(t^d)$:
here $wu-quw=v^d$, while $uB_q(f)$ is spanned by the PBW monomials $u^av^bw^c$ with $a\geq1$, so $wu$ lies in $B_q(f)u$ but not in $uB_q(f)$. Giving $w$ the weight $s=d$ pushes $f$ down one filtration degree instead, and $\gr B_q(f)$ is a quantum affine space, where $U$ is normal. Nowhere else does the choice matter.

Finally, for $0\neq a\in B_q(f)$ write $\deg(a)=\min\{N : a\in F_N\}$ and let $\lf(a)$ be the image of $a$ in $F_{\deg(a)}/F_{\deg(a)-1}$. Since $\gr B_q(f)$ is a domain,
\begin{align*}
  \lf(ab) &= \lf(a)\lf(b), & \deg(ab) &= \deg(a)+\deg(b).
\end{align*}

\section{ A Grading on  $B_q(f)$}

\begin{lemma}\label{lem:grad}
$B_q(f)$ is $\mathbb{Z}/e\times\mathbb{Z}/e$-graded by $|u|=(1,0)$, $|v|=(0,1)$,
$|w|=(-1,-1)$.
\end{lemma}

\begin{proof}
The relation $uv-qvu$ is homogeneous of degree $(1,1)$. In $wu-quw-f(v)$ the first two terms have degree $(0,-1)$, while the summand $v^j$ of $f(v)$ has degree $(0,j)$; so homogeneity amounts to $j\equiv-1\pmod e$ for every $j\in\supp(f)$, which is what $e$ was defined to give. The third relation is symmetric.
\end{proof}

We write $\mu_e\times\mu_e$ for the group of grading automorphisms, so that
$(\xi_1,\xi_2)$ acts by
\begin{align*}
  u&\mapsto \xi_1 u, & v&\mapsto \xi_2v, & w&\mapsto (\xi_1\xi_2)^{-1}w.
\end{align*}

\begin{proposition}\label{prop:norm}
Suppose $e>1$ and let $m=n/e$. Then $u^m$ and $v^m$ are normal but not central:
\begin{align*}
  wu^m &= q^mu^mw, & u^mv &= q^mvu^m, & wv^m &= q^{-m}v^mw, & v^mu &= q^{-m}uv^m.
\end{align*}
\end{proposition}

\begin{proof}
For every \(j\in\operatorname{supp}(f)\),   \(e\mid(j+1)\) and \(e\mid n\). In particular, \(e\mid\gcd(n,j+1)\). Set \(g_j=\gcd(n,j+1)\). Since \(e\mid g_j\), we may write \(g_j=er_j\) for some positive integer \(r_j\). Since \(g_j\mid n=em\), it follows that \(r_j\mid m\), and therefore
\[
\frac{n}{\gcd(n,j+1)}
=
\frac{em}{er_j}
=
\frac{m}{r_j}
\mid m.
\]

Since \(q\) is a primitive \(n\)-th root of unity,
\[
\operatorname{ord}(q^{j+1})
=
\frac{n}{\gcd(n,j+1)}.
\]
Thus, \(\operatorname{ord}(q^{j+1})\mid m\), and hence \((q^{j+1})^m=1\). Moreover, the assumption \(n\nmid(j+1)\) implies \(q^{j+1}\neq1\). Consequently,
\[
[m]_{q^{j+1}}
=
\frac{1-(q^{j+1})^m}{1-q^{j+1}}
=
0.
\]
Likewise, \((q^{-(j+1)})^m=1\) and \(q^{-(j+1)}\neq1\), so
\[
[m]_{q^{-(j+1)}}
=
\frac{1-(q^{-(j+1)})^m}{1-q^{-(j+1)}}
=
0.
\]

Applying \cite[Lemma 3.1]{GY} with \(k=m\), we obtain
\[
wu^m=q^mu^mw,
\qquad
wv^m=q^{-m}v^mw.
\]
On the other hand, the relation \(uv=qvu\) yields, by induction, \(u^mv=q^mvu^m\), while \(vu=q^{-1}uv\) gives \(v^mu=q^{-m}uv^m\). Together with the  relations \(u^mu=uu^m\) and \(v^mv=vv^m\), these identities show that \(u^m\) and \(v^m\) are normal elements of \(B_q(f)\).

It remains to show that they are not central. Since \(e>1\) and \(m=\frac{n}{e}\), we have \(0<m<n\). As \(q\) has order exactly \(n\), \(q^m\neq1\). Therefore,
\[
u^mv=q^mvu^m\neq vu^m,
\]
so \(u^m\notin Z(B_q(f))\). Similarly, \(q^{-m}\neq1\), and
\[
v^mu=q^{-m}uv^m\neq uv^m,
\]
hence \(v^m\notin Z(B_q(f))\).
\end{proof}

\begin{lemma}\label{lem:deg0}
$Z(B_q(f))$ lies in the component of degree $(0,0)$.
\end{lemma}

\begin{proof}
The generators of $X$ are homogeneous, so the grading extends to $\widehat{B}$. The summands $u^{-1}v^j$ and $v^{-1}u^j$ of $\gamma$ have degrees $(-1,j)$ and $(j,-1)$; since $e\mid(j+1)$, both equal $(-1,-1)$, which is the degree of $w$. Hence $w'$ is homogeneous of degree $(-1,-1)$.

Now $\widehat{B}$ is $\mathbb{Z}^3$-graded by the exponents of $u,v,w'$, with
one-dimensional components, so $Z(\widehat{B})$ is spanned by the central monomials. For $M=u^av^bw'^c$ the relations $uv=qvu$, $w'u=quw'$, $w'v=q^{-1}vw'$ give
\begin{align*}
  Mu &= q^{\,c-b}uM, & Mv &= q^{\,a-c}vM, & Mw' &= q^{\,b-a}w'M,
\end{align*}
so $M$ is central if and only if $a\equiv b\equiv c\pmod n$. As $e\mid n$, in that case $|M|=(a-c,\,b-c)=(0,0)$. So $Z(\widehat{B})$ lies in degree $(0,0)$, and hence so does
$Z(B_q(f))=Z(\widehat{B})\cap B_q(f)$.
\end{proof}

\section{The ozone groups}

\begin{lemma}\label{lem:diag}
Let $\varphi\in\Aut(B_q(f))$ satisfy $\varphi(u^n)=u^n$ and $\varphi(v^n)=v^n$. Then $\varphi(u)=\xi_1u$ and $\varphi(v)=\xi_2v$ with $\xi_1^n=\xi_2^n=1$.
\end{lemma}

\begin{proof}
Multiplicativity of the degree gives $n\deg\varphi(u)=n$, so $\deg\varphi(u)=1$, and by Lemma \ref{lem:filt}
\begin{align*}
  \varphi(u) &= a_1u+a_2v+a_3w+b, & \lf(\varphi(u)) &= a_1U+a_2V+a_3W,
\end{align*}
where $a_3=0$ unless $s=1$. Also $\lf(\varphi(u))^n=\lf(u^n)=U^n$. The element $U$ is normal in $\gr B_q(f)$ and
\begin{align*}
  \gr B_q(f)/(U) &\cong \kk\langle V,W\rangle/(WV-q^{-1}VW),
\end{align*}
a skew polynomial ring in two variables, hence a domain. Reducing $\lf(\varphi(u))^n=U^n$ modulo $(U)$ gives $(a_2V+a_3W)^n=0$, so $a_2=a_3=0$. Thus $\lf(\varphi(u))=a_1U$ with $a_1^n=1$, and $\varphi(u)=a_1u+b$ with $b\in\kk$. Expanding $(a_1u+b)^n=u^n$ and comparing coefficients of $u^{n-1}$ gives $na_1^{n-1}b=0$, so $b=0$ in characteristic zero. The argument for $v$ is the same.
\end{proof}

\begin{lemma}\label{lem:lau}
$B_q(f)\cap\kk_q[u^{\pm1},v^{\pm1}]=\kk_q[u,v]$.
\end{lemma}

\begin{proof}
$\widehat{B}$ is free over $\kk_q[u^{\pm1},v^{\pm1}]$ on $\{w'^{\,c}\}_{c\geq0}$, and $w=w'+\gamma$ with $\gamma\in\kk_q[u^{\pm1},v^{\pm1}]$. So for
$x=\sum\lambda_{abc}u^av^bw^c$ in $B_q(f)$ with $C=\max\{c:\lambda_{abc}\neq0\}$, the component of $x$ in $w'$-degree $C$ is
$\bigl(\sum_{a,b}\lambda_{abC}u^av^b\bigr)w'^{\,C}\neq0$. If $x$ lies in
$\kk_q[u^{\pm1},v^{\pm1}]$ then $C=0$, so $x\in\kk_q[u,v]$.
\end{proof}

\begin{proposition}\label{prop:key}
Let $\varphi\in\Aut(B_q(f))$ fix $u^n$, $v^n$ and $\Omega$. Then there are
$\xi_1,\xi_2\in\mu_e$ with
\begin{align*}
  \varphi(u) &= \xi_1u, & \varphi(v) &= \xi_2v, & \varphi(w) &= (\xi_1\xi_2)^{-1}w,
\end{align*}
so $\varphi$ is a grading automorphism for the grading of Lemma \ref{lem:grad}.
\end{proposition}

\begin{proof}
Write $\varphi(u)=\xi_1u$, $\varphi(v)=\xi_2v$ as in Lemma \ref{lem:diag}; then $\varphi$ extends to $\widehat{B}$, as noted in Section \ref{sec:prelim}. Since $\Omega=uvw'$ and $uv$ is invertible in the domain $\widehat{B}$,
\begin{align*}
  uvw' &= \varphi(\Omega) = \xi_1\xi_2\,uv\,\varphi(w'), & \varphi(w') &= \zeta w',
  & \zeta &= (\xi_1\xi_2)^{-1}.
\end{align*}
Writing $w=w'+\gamma$, and using $\varphi(u^{-1})=\xi_1^{-1}u^{-1}$ and
$\varphi(v^{-1})=\xi_2^{-1}v^{-1}$,
\begin{align*}
  \varphi(w)-\zeta w &= \varphi(\gamma)-\zeta\gamma\\
  &= \sum_{j\in\supp(f)}\alpha_jc_j\bigl(\xi_1^{-1}\xi_2^{\,j}-\zeta\bigr)u^{-1}v^j
   -\sum_{j\in\supp(f)}\beta_jc_j\bigl(\xi_2^{-1}\xi_1^{\,j}-\zeta\bigr)v^{-1}u^j.
\end{align*}
The left-hand side lies in $B_q(f)$ and the right-hand side in
$\kk_q[u^{\pm1},v^{\pm1}]$, so both lie in $\kk_q[u,v]$ by Lemma \ref{lem:lau}. The
monomials $u^{-1}v^j$ and $v^{-1}u^j$ have pairwise distinct exponent vectors $(-1,j)$ and
$(j,-1)$, and none of them lies in $\kk_q[u,v]$. Since $\alpha_j,\beta_j,c_j\neq0$ for
$j\in\supp(f)$, every coefficient vanishes: from
$\xi_1^{-1}\xi_2^{\,j}=\zeta=(\xi_1\xi_2)^{-1}$ we get $\xi_2^{\,j+1}=1$, from
$\xi_2^{-1}\xi_1^{\,j}=\zeta$ we get $\xi_1^{\,j+1}=1$, and $\varphi(w)=\zeta w$. With
$\xi_1^n=\xi_2^n=1$ this puts $\xi_1,\xi_2$ in $\mu_e$, so $\varphi$ is the grading
automorphism $(\xi_1,\xi_2)$.
\end{proof}

\begin{theorem}\label{thm:main}
\begin{align*}
  \Oz(B_q(f)) &\cong \mu_e\times\mu_e,
\end{align*}
acting as in Proposition \ref{prop:key}. In particular, the group is trivial if and only
if $e=1$.
\end{theorem}

\begin{proof}
An ozone automorphism fixes the central elements $u^n$, $v^n$ and $\Omega$. By Proposition~\ref{prop:key}, $\varphi$ is a grading automorphism. Conversely, let $(\xi_1,\xi_2)\in\mu_e\times\mu_e$. The corresponding grading automorphism acts on $B_{(a,b)}$ by multiplication by
$\xi_1^a\xi_2^b$. Since $Z(B_q(f))\subseteq B_{(0,0)}$ by
Lemma~\ref{lem:deg0}, it fixes the center pointwise. Hence it is an
ozone automorphism.
\end{proof}

\begin{corollary}\label{cor:detect}
\begin{align*}
  \Oz(B_q(f)) &= \{\varphi\in\Aut(B_q(f)) :
  \varphi(u^n)=u^n,\ \varphi(v^n)=v^n,\ \varphi(\Omega)=\Omega\}.
\end{align*}
\end{corollary}

\begin{proof}
One inclusion holds because $u^n,v^n,\Omega$ are central; the other combines Proposition
\ref{prop:key} with Lemma \ref{lem:deg0}.
\end{proof}

\begin{remark}
The center of $B_q(f)$ at a root of unity is not known \cite[Question 5.3]{GY}. Corollary \ref{cor:detect} does not require an explicit description of the center: the elements $u^n$, $v^n$ and $\Omega$
suffice.
\end{remark}

\begin{corollary}\label{cor:normal}
Let $N$ be the multiplicative monoid of regular normal elements of $B_q(f)$, and $m=n/e$.
Say $a\sim b$ when $z_1a=z_2b$ for nonzero central $z_1,z_2$, as in
\cite[(E1.10.1)]{CGWZ2}. Then
\begin{align*}
  N/\!\sim\ &\cong\ \mu_e\times\mu_e,
\end{align*}
generated by the classes of $u^m$ and $v^m$. In particular every normal element of
$B_q(f)$ is central if and only if $e=1$.
\end{corollary}

\begin{proof}
For $a\in N$, let
\[
\eta_a(b)=a^{-1}ba.
\]
Since \(a\) is normal, \(\eta_a\in\operatorname{Aut}(B_q(f))\). Moreover, for every \(z\in Z(B_q(f))\),
\[
\eta_a(z)=a^{-1}za=z.
\]

Thus $\eta_a\in\Oz(B_q(f))$.

By \cite[Lemma 1.9(1)]{CGWZ2}, the map
\[
N\longrightarrow\Oz(B_q(f)),\qquad a\longmapsto\eta_a,
\]
is surjective, since $B_q(f)$ is a prime PI algebra that is module-finite
over its center. Moreover,
\[
\eta_a=\eta_b
\]
if and only if $a\sim b$; hence
\[
N/\!\sim\ \cong\Oz(B_q(f)).
\]
By Theorem~\ref{thm:main},
\[
N/\!\sim\ \cong\mu_e\times\mu_e.
\]

It remains to identify generators. Put $m=n/e$. By
Proposition~\ref{prop:norm},
\[
\eta_{u^m}=(1,q^{-m}),
\qquad
\eta_{v^m}=(q^m,1).
\]
Since $q$ has order $n$,
\[
\operatorname{ord}(q^m)
=\frac{n}{\gcd(n,m)}
=\frac{n}{m}
=e.
\]
Hence $(1,q^{-m})$ and $(q^m,1)$ generate
$\mu_e\times\mu_e$. Therefore the classes of $u^m$ and $v^m$
generate $N/\!\sim$.

Finally, $N/\!\sim$ is trivial if and only if $e=1$. Since the kernel of
$a\mapsto\eta_a$ consists precisely of the central elements, this is
equivalent to every normal element of $B_q(f)$ being central.
\end{proof}

\begin{proposition}\label{prop:phi}
With the notation of Section \ref{sec:prelim},
\begin{align*}
  \Oz(B_q(f)) &= \{\phi_{a,\xi,0} : a\in\mu_e \text{ and } \xi a\in\mu_e\},
\end{align*}
and $\phi_{a,\xi,0}$ is the grading automorphism $(a,\xi a)$.
\end{proposition}

\begin{proof}
Given $(\xi_1,\xi_2)\in\mu_e\times\mu_e$, set $a=\xi_1$ and $\xi=\xi_2\xi_1^{-1}$. Since
$e\mid(i+1)$ for $i\in\supp(f)$ we get $\xi_1^{\,i+1}=\xi_2^{\,i+1}=1$, hence
$\xi^{i+1}=1$, and $e\mid(d+1)$ gives $a^{d+1}=1$, so $a^{d-i}=a^{(d+1)-(i+1)}=1$. Thus
$\phi_{a,\xi,0}$ is defined and
\begin{align*}
  \xi^{-1}a^{d-1} &= \xi^{-1}a^{-2} = (\xi_1\xi_2)^{-1},
\end{align*}
so it is the grading automorphism $(\xi_1,\xi_2)$. With Theorem \ref{thm:main} this gives
both inclusions.

The converse can also be read off from $\Omega$ alone, without appealing to Theorem
\ref{thm:main}; this is the computation of \cite[Theorem 3.13]{GY}. Suppose that \(\varphi=\varphi_{a,\xi,h}\) fixes the central elements. From $\phi(u^n)=a^nu^n$ and
$\phi(v^n)=(\xi a)^nv^n$ we get $a^n=(\xi a)^n=1$, and
\begin{align*}
  \phi(\Omega) = a^{d+1}uvw+a^2\xi\,uvh
   +\sum_{j\in\supp(f)}\beta_jc_ja^{j+1}u^{j+1}-q\sum_{j\in\supp(f)}\alpha_jc_j(\xi a)^{j+1}v^{j+1}.
\end{align*}
The part of $\phi(\Omega)-\Omega$ of $w$-degree zero lies in $\kk_q[u,v]$; every monomial
of $uvh$ has positive degree in both $u$ and $v$, while $u^{j+1}$ and $v^{j+1}$ do not, so
the two contributions vanish separately. Hence $h=0$ together with
\begin{align*}
  a^{j+1} &= 1, & (\xi a)^{j+1} &= 1 && (j\in\supp(f)),
\end{align*}
which with $a^n=(\xi a)^n=1$ puts $a$ and $\xi a$ in $\mu_e$.
\end{proof}

\begin{example}\label{ex:small}
Take $n=4$, so $q=i$, and $f(t)=t$. Then $\supp(f)=\{1\}$ and $4\nmid2$, while
$e=\gcd(4,2)=2$ and $m=2$. Here $[2]_{q^2}=1+q^2=0$, so $u^2$ is normal and not central,
and
\begin{align*}
  \eta:\quad u&\mapsto u, & v&\mapsto -v, & w&\mapsto -w
\end{align*}
is a nontrivial ozone automorphism: it fixes $u^4$, $v^4$ and
$\Omega=uvw+\tfrac{i}{2}u^2+\tfrac12v^2$. In the notation of Proposition \ref{prop:phi}
this is $\phi_{1,-1,0}$, and $\Oz(B_i(t))\cong\mu_2\times\mu_2$.
\end{example}

\begin{example}\label{ex:six}
Fix $n=6$ and compare three polynomials. For $f=t$ we get $e=\gcd(6,2)=2$ and $m=3$; for
$f=t^2$ we get $e=\gcd(6,3)=3$ and $m=2$, so
\begin{align*}
  \Oz(B_q(t)) &\cong \mu_2\times\mu_2, & \Oz(B_q(t^2)) &\cong \mu_3\times\mu_3.
\end{align*}

For $f=t+t^2$, though, $e=\gcd(6,2,3)=1$, so the ozone group is trivial and by Corollary
\ref{cor:normal} every normal element is central. The candidate $u^3$ fails:
\begin{align*}
  wu^3 &= q^3u^3w+[3]_{q^2}\,vu^2+[3]_{q^3}\,v^2u^2 = -u^3w+v^2u^2,
\end{align*}
because $q^2$ has order $3$ and $q^3$ has order $2$, so $[3]_{q^2}=0$ but $[3]_{q^3}=1$.
The summand $t$ kills its term and $t^2$ does not, which is what taking a gcd over the
whole support amounts to.
\end{example}

\begin{remark}\label{rem:cgwz}
For $f=t^2$ the algebra is the nodal cubic algebra $B_q$ of \cite[(E1.5.1)]{CGWZ2}. Their
generators $x,y,z$ are our $u,v,w$, and we keep $u,v,w$ below; with that dictionary
Theorem \ref{thm:main} returns \cite[Proposition 1.8]{CGWZ2}.

The two arguments need different things. The one in \cite{CGWZ2} runs over graded
automorphisms of a quadratic AS-regular algebra, so a degree count places $\varphi(u)$ in
$\kk u+\kk v$ immediately; the corresponding step here is Lemma \ref{lem:diag}, which
replaces the degree count by leading forms and so covers members like $f=t$, where there
is no such grading.

Their Lemma 1.7 and Proposition 1.8 both carry the restriction $n\neq1,3$. For $f=t^2$ our
standing hypothesis $n\nmid(j+1)$ reads $n\nmid3$, which is the same restriction: what
looks there like a technical exclusion is exactly the condition that makes $\alpha_2$,
$\beta_2$, and therefore $\Omega$, exist. The case $n=2$ is treated separately in
\cite[Example 1.6]{CGWZ2}; here it needs no separate treatment, since $e=\gcd(2,3)=1$ and
Theorem \ref{thm:main} already gives a trivial group. And the $3$ in the answer is $j+1$
for the single $j\in\supp(t^2)$.

In \cite[Lemma 1.7]{CGWZ2} the center of $B_q$ is described using $z^n$, that is $w^n$. That element is not always central: \cite[Remark 3.7]{GY} works out $n=4$, where $w^4+2(1-q)\Omega w$ is central and $w^4$ is not. The ozone group is a different matter, and Corollary \ref{cor:detect} shows it: only $u^n$, $v^n$ and $\Omega$ go into determining it, so Proposition 1.8 stands.
\end{remark}

\begin{remark}
Every $B_q(f)$ is Calabi--Yau \cite[Theorem A]{GY}, so the algebras with $e>1$ give an infinite family of Calabi--Yau algebras with nontrivial ozone group. The groups obtained are abelian, in agreement with \cite{LWZ} for the AS-regular members.
\end{remark}

\begin{proposition}\label{prop:rank}
$\rk_{Z(B_q(f))}B_q(f)=n^2$ and $B_q(f)$ has PI degree $n$. Consequently
\begin{align*}
  |\Oz(B_q(f))| &= e^2, & \rk_Z\big/|\Oz| &= m^2.
\end{align*}
\end{proposition}

\begin{proof}
By Section \ref{sec:prelim}, $\widehat{B}$ is a central localization of $B_q(f)$, so the fraction field of the center is the same on both sides, and so is the rank. It therefore suffices to work in $\widehat{B}$.

The relations $w'u=quw'$ and $w'v=q^{-1}vw'$ give $w'uv=uvw'$, so
\begin{align*}
  \Omega^c &= (uv)^cw'^{\,c} \qquad (c\geq0).
\end{align*}
Now $\{w'^{\,c}\}_{c\geq0}$ is a basis of $\widehat{B}$ over $\kk_q[u^{\pm1},v^{\pm1}]$
and $uv$ is a unit, so $\{\Omega^c\}_{c\geq0}$ is one as well. Since $\Omega$ is central,
\begin{align*}
  \widehat{B} &= \kk_q[u^{\pm1},v^{\pm1}][\Omega],
\end{align*}
a polynomial ring in a central variable over the quantum torus in two variables. That torus is $\mathbb{Z}^2$-graded with one-dimensional components, so its center is spanned by the central monomials, and $u^av^b$ commutes with $u$ and $v$ exactly when $q^a=q^b=1$. Therefore
\begin{align*}
  Z(\widehat{B}) &= \kk[u^{\pm n},v^{\pm n},\Omega],
  & \widehat{B} &= \bigoplus_{0\leq a,b\leq n-1}Z(\widehat{B})\,u^av^b,
\end{align*}
which gives $\rk_{Z(B_q(f))}B_q(f)=n^2$. The PI degree is the square root of the rank \cite[(E3.6.1)]{CGWZ2}. The rest follows from Theorem \ref{thm:main} and $e\mid n$.
\end{proof}

\begin{remark}
The computation above never uses AS-regularity, so it covers every member of the family. For the AS-regular members \cite[Theorem E]{CGWZ2} says that $|\Oz(A)|$ divides $\rk_Z(A)$; here the divisibility holds throughout the family, and the quotient is computed exactly, namely $m^2$. Since $e<n$, the divisibility is always strict.
\end{remark}

\begin{remark}
No $B_q(f)$ satisfying the standing hypothesis is isomorphic to a skew polynomial ring. Suppose $B_q(f)\cong S_{\mathbf p}$. The order of the ozone group and the rank over the center are invariants of the algebra, so it is enough to compare them on the two sides. For a PI skew polynomial ring the center is the fixed ring $S_{\mathbf p}^{\mathcal O}$ under the group $\mathcal O=\Oz(S_{\mathbf p})$ generated by the conjugations $\eta_{x_i}$ \cite[Hypothesis 0.3 and \S1]{CGWZ1}, and $S_{\mathbf p}$ is noetherian PI AS-regular, so \cite[Proposition 3.9(2)]{CGWZ2} gives $|\Oz(S_{\mathbf p})|=\rk_Z(S_{\mathbf p})$. On the other side $|\Oz(B_q(f))|=e^2<n^2=\rk_Z(B_q(f))$ by Theorem \ref{thm:main} and
Proposition \ref{prop:rank}.

We restrict application of the skew polynomial machinery deliberately to $S_{\mathbf p}$. The characterization \cite[Theorem F]{CGWZ2} requires the algebra to be AS-regular and generated in degree one, conditions which are not satisfied by $B_q(f)$ for general $f$; \cite{CGWZ2} remark that there is no known intrinsic characterization of $S_{\mathbf p}$ as an ungraded algebra. 
\end{remark}

\begin{remark}
Corollary \ref{cor:normal} is about classes, not elements. Normal elements outside the submonoid generated by $u^m$, $v^m$ and the center do exist: for $n=6$ and $f=t^2$ one has $q+q^{-2}=0$ and $q^{-1}+q^2=0$, so $w^2u=q^2uw^2$ and $w^2v=q^{-2}vw^2$, and $w^2$ is normal. Suppose $w^2=z\,u^{2a}v^{2b}$ with $z$ central. Leading forms of central elements are central in $\gr B_q(t^2)$, because $\lf$ is multiplicative and $\gr B_q(t^2)$ is a domain, so $\lf(z)$ commutes with $U$, $V$ and $W$. Comparing leading forms in the filtration of Lemma \ref{lem:filt}, where $s=2$, forces $a=b=0$ and $\lf(z)=W^2$; but
$W^2U=q^2UW^2$ with $q^2\neq1$. Its class is $\eta_{u^m}^{-1}\eta_{v^m}^{-1}$, so the ozone group sees nothing new.
\end{remark}

\begin{question}
Lemma \ref{lem:deg0} places $Z(B_q(f))$ in degree $(0,0)$ but does not describe it, and \cite[Question 5.3]{GY} remains open: which elements of $Z(\widehat{B})$ actually lie in $B_q(f)$? The delicate element is $w'^{\,n}$. The mechanism is visible in \cite[Remark 3.7]{GY}, where $w^4$ fails to be central although $w^4+2(1-q)\Omega w$ is.
\end{question}

\subsection*{Declaration of AI-assisted technologies}
During the preparation of this manuscript, the authors used Claude (Anthropic) to improve the English language and readability of the exposition and to assist with the preparation and refinement of the LaTeX source. 

\end{document}